\documentclass[11pt,letterpaper]{article}
\usepackage[margin=1.5in]{geometry}
\usepackage{amsmath,times}
\usepackage{amssymb,amsmath}
\usepackage{graphicx}
\usepackage{amsthm}
\usepackage{multicol}
\usepackage[all]{xy}
\usepackage{cancel}
\usepackage{tikz}
\usepackage[utf8]{inputenc}
\usepackage{url}
\usepackage{authblk}
\usepackage{booktabs}
\usepackage{xcolor}
\usepackage[dvipsnames]{xcolor}
\usepackage{tcolorbox}
\tcbuselibrary{theorems}
\usetikzlibrary{tikzmark,positioning,cd,arrows.meta}
\usepackage{graphicx}
\usepackage{multicol}
\usepackage{tikz}
\usepackage{doi}
\usepackage{tikz-cd}
\usepackage{orcidlink}
\usepackage{url}
\usepackage{authblk}
\usepackage{booktabs}
\usepackage{xcolor}
\usepackage[dvipsnames]{xcolor}
\usepackage{tcolorbox}
\tcbuselibrary{theorems}
\usetikzlibrary{tikzmark,positioning,cd,arrows.meta}

\title{Finite generating sets for monoids of $G$-equivariant functions}

\author[1]{Ram\'on H. Ruiz-Medina \orcidlink{0000-0003-2916-9160}}
\author[2]{Victor M. Lara-Gómez \orcidlink{0000-0000-0000-0000}}
\author[3]{Gerardo Romero-Rosales \orcidlink{0000-0000-0000-0000}}

\affil[1]{Center for Mathematical Sciences.
National Autonomous University of Mexico.
Morelia, Michoac\'an, Mexico. CP 58089 \\ \texttt{harath.ruiz@academicos.udg.mx}}

\affil[2]{University Center of Exact Sciences and Engineering, University of Guadalajara, Guadalajara, Jalisco, Mexico \\ \texttt{victor.lara8714@alumnos.udg.mx}}
\affil[3]{University Center of Exact Sciences and Engineering, University of Guadalajara, Guadalajara, Jalisco, Mexico \\ \texttt{gerardo.romero2895@alumnos.udg.mx}}
\date{}

\newtheorem{theorem}{Theorem}[]
\newtheorem{lemma}[theorem]{Lemma}

\newtheorem{corollary}[theorem]{Corollary}
\newtheorem{proposition}[theorem]{Proposition}
\newtheorem{example}[theorem]{Example}
\newtheorem{definition}[theorem]{Definition}

\newtheorem{claim}[theorem]{Claim}

\newcommand{\zz}{\mathbb{Z}}

\newcommand{\rank}{\mathrm{Rank}}
\newcommand{\Sym}{\mathrm{Sym}}

\newcommand{\EndG}{\mathrm{End}_{G}(X)}
\newcommand{\AutG}{\mathrm{Aut}_{G}(X)}
\newcommand{\Hbox}{\mathcal{B}_{[H]}}
\newcommand{\Kbox}{\mathcal{B}_{[K]}}

\newcommand{\StabG}{\mathrm{Stab}_{G}(X)}
\newcommand{\ConjG}{\mathrm{Conj}(G)}
\newcommand{\ConjX}{\mathrm{Conj}_{G}(X)}

\newcommand{\EndH}{\mathrm{End}_{G}(\Hbox)}

\newcommand{\AutH}{\mathrm{Aut}_{G}(\Hbox)}

\newcommand{\ibox}{\mathcal{B}_{i}}

\begin{document}

\maketitle

\begin{abstract}
Given the action of a group $G$ on a set $X$, the set of all $G$-equivariant functions, i.e., those satisfying $f(g\cdot x)=g\cdot f(x)$ for all $g\in G$ and $x\in X$, forms a monoid under composition. In this work we study their generating sets. First, we propose bounds for the cardinalities of the generating sets of their group of units, denoted by $\AutG$. Subsequently, using so-called orbital infiltrations, certain transformations that turn out to be indispensable and provide relevant structural information about the monoid, we determine conditions on the group $G$, the set $X$, and the action that prevent the whole monoid $\EndG$ from admitting a finite generating set. \\

\textbf{Keywords:} $G$-equivariant function, group of units, generating set. \\

\textbf{MSC 2020:} 20B25, 20E22, 20M20, 20B07.
\end{abstract}

\section{Introduction}  

$G$-sets and their $G$-equivariant transformations constitute a natural bridge between group theory, algebraic topology, and transformation semigroups. The monoid $\EndG$, consisting of all $G$-equivariant functions from $X$ to itself, and its group of units $\AutG$, which comprises the equivariant bijections, are objects that have been studied from multiple structural perspectives: regularity, generation by idempotents, and embedding properties into wreath products, among others \cite{cita21,cita22,cita23,cita24,cita13,cita14}. Although the classical literature has laid solid foundations for free $G$-sets and independence algebras \cite{cita11,cita12,cita15}, questions concerning finite generation and rank (the minimum number of generators) of these structures have received a still fragmented treatment, leaving open the question of how these invariants behave under arbitrary actions and how they relate to each other.

The present work addresses two problems of equal relevance affecting, respectively, the group of units and the whole monoid. On the one hand, we aim to compute the rank of $\AutG$ in explicit terms of the action, unraveling how the number of orbits and the structure of stabilizers determine the generating complexity of the equivariant symmetry group. On the other hand, with equal weight, we investigate the conditions on the set $X$, the group $G$, and the action itself that force $\EndG$ *not* to be finitely generated. This second problem, although seemingly more global, does not reduce to the first: a finitely generated group of units does not guarantee finite generation of the whole monoid, and it is precisely this tension that we aim to elucidate.

As our main contributions, we offer a complete characterization of the rank of $\AutG$ and, simultaneously, establish necessary and sufficient conditions for $\EndG$ to be finitely generated or not. These conditions reveal a precise dichotomy: the presence of infinitely many orbits, infinite stabilizers, or certain torsion-free subgroups of $G$ forces non-finite generation, while their absence leads to finitely generated monoids. Our results extend and unify recent work on cellular automata and finite $G$-sets \cite{castillo2016a,castillo2016b,castillo2017,castillo2021,paper}, and complement contemporary developments on monoids of $G$-equivariant functions \cite{paper2,paper3,chap}, thus offering a comprehensive panorama that clearly distinguishes generatable scenarios from those that are not.

For any natural number $n\in \mathbb{N}$ we denote an $n$-element set by $[n]=\{1,2,3,...,n\}$. We denote the set of all conjugacy classes of subgroups of $ G $ by $ \ConjG $, and denote them as follows: $ [H]:=\{g^{-1}Hg:\ g\in G\} $. In the cases where there is a countable number of conjugacy classes, we denote by $ [H_{1}], [H_{2}],...,[H_{r}],... $ all the conjugacy classes of subgroups of $ G $, ordered by their cardinality as:
$$  |H_{1}| \leq |H_{2}| \leq \dots \leq |H_{r}|\leq \dots. $$ 
We also denote a finite set with $r$ elements by $[r]=\{1,2,...,r\}$. We can define a partial order on $\ConjG$ given by:
$$H \leq K \iff \exists g\in G\ \text{such that}\ H\leq g^{-1}Kg. $$

Given the action of a group $ G $ on a set $ X $, we recall the $ G $-orbits and the stabilizer of elements in $ X $ as follows: for $ x \in X $ :
$$  Gx := \{ g \cdot x \mid g \in G \},\  G_{x} := \{ g \in G \mid g \cdot x = x \}. $$ 
Subsequently, based on the stabilizer, we define the following sets in $ X $. Given $ H \leq G $, let:
$$  \mathcal{B}_{H} := \{ x \in X \mid G_{x} = H \}, $$ 
$$  \Hbox := \{ x \in X \mid [G_{x}] = [H] \}. $$ 
We can extend the partial order of conjugacy classes to these sets as 
$$\Hbox \leq \Kbox \iff [H]\leq [K].$$
Let $ X/G $ and $ \Hbox/ G $ denote the set of orbits of the action of $ G $ on $ X $ and $ \Hbox $ respectively. In the case that the conjugacy classes are indexed by a natural number, we simplify the notation of the sets $ \mathcal{B}_{[H_{i}]} $ as $ \mathcal{B}_{[H_{i}]}=\ibox $. When possible, the index or the subgroup will be omitted from the notation. 

Note that some conjugacy classes may not be included in the set of stabilizers of the action of $ G $ on $ X $; we then define the set of subgroups of $G$ that act as stabilizers for elements in $X$ as:
$$  \StabG:=\{G_{x}|\ x\in X\}. $$

However, we must note that if $ H \in \StabG $, the complete conjugacy class of $ H $, $ [H] $, is contained in $ \StabG $, because if $ h \in G_{x} $ for some $ x \in X $, it holds that 
$$  (g^{-1}hg) \cdot (g^{-1}\cdot x) = g^{-1}\cdot x, \ \forall g \in G, $$ 
meaning that any conjugate of $ h $ stabilizes at least one element in $ X $. We then denote by $\ConjX$ the set of conjugacy classes of subgroups of $G$ in $\StabG$,
$$\ConjX:=\{[H]\in \ConjG|\ H\in \StabG\}.$$

The following result is well known in the theory of $ G $-equivariant functions. 
\begin{lemma}\label{lema1}
Let $ G $ be a group acting on a set $ X $. Given $ x,y \in X $, the following holds: 

\begin{enumerate} 
\item[i)] There exists a $ G $-equivariant function $ \tau \in \EndG $ such that $ \tau(x)=y $ if and only if $ G_{x} \leq G_{y} $.
\item[ii)] There exists a bijective $ G $-equivariant function $ \sigma \in \AutG $ such that $ \sigma(x)=y $ if and only if $ G_{x} = G_{y} $.
\end{enumerate}
\end{lemma}
\noindent Further details on this result can be found in \cite{paper}.\\

Based on these results, given $x,y\in X$ such that $G_{x}\leq G_{y}$, we define the following $G$-equivariant functions:
$$  [x\mapsto y](z)= \left\{  \begin{array}{cc}
    g\cdot y &  z=g\cdot x, \\
    z & \text{otherwise.}
\end{array} \right. $$ 
We call these functions elementary collapsing of type $(G_{x},[G_{y}]_{N_{G_{x}}})$. We also define the following $G$-equivariant and bijective functions:
$$  (x\leftrightarrow y)(z)= \left\{  \begin{array}{cc}
    g\cdot y &  z=g\cdot x, \\
    g\cdot x &  z=g\cdot y, \\
    z & \text{otherwise.}
\end{array} \right. $$ 

Note that if $Gx \neq Gy$, then $[x\mapsto y]$ is neither injective nor surjective, while, if $Gx=Gy$ and $G$ is a finite group (as will be considered in later cases), then these functions are bijective and will be denoted as $(x\mapsto y)$. In this case, there exists an element $g\in G$ such that $y=g\cdot x$; we want to define $G$-equivariant functions 
$$(x\rightarrow g\cdot x)(z):= \left\{  \begin{array}{cc}
    hg\cdot x &  z=h\cdot x, \\
    z & \text{otherwise.}
\end{array} \right.$$

It is not hard to verify that this function is well-defined if and only if $g\in N_{G}(G_{x})/G_{x}$, and that moreover
$$(x\rightarrow g\cdot x)^{-1}=(x\rightarrow g^{-1}\cdot x).$$

Given a subgroup $ H\leq G $ and a subset $ N\subseteq G$, we define the $ N $-conjugacy classes of $ H $ as:
$$  [H]_{N}:=\{n^{-1}Hn:\ n\in N\}. $$ 
It is easy to see that $ [H]_{N} \subseteq [H] $, meaning that the elements in an $ N $-conjugacy class of $ H $ are some of the conjugate subgroups of $ H $, specifically those given by conjugating elements in $ N $. We denote the normalizer of a subgroup $ H $ simply as $ N_{G}(H)=N_{H} $ and when $H\leq G$ is a numbered subgroup of $G$, i.e., $H_{i}$, we denote the normalizer of $H_{i}$ simply as $N_{G}(H_{i})= N_{i}$.\\

\begin{lemma} 
For a finite $G$-set $X$ and for any $H\in \StabG$, the following statements hold. 
\begin{enumerate} 
\item If $X$ is a finite $G$-set, then $\EndG$ is finite. 
\item The index $[G:H]$ is finite. 
\item The quotient $N_{G}(H)/H$ is finitely generated.
\end{enumerate}
\end{lemma}

The following is an important result on the structure of the group of units of the monoid of $G$-equivariant functions presented in \cite{paper}.
\begin{theorem}
Given a group $G$ acting on a set $X$, it holds that
$$\AutG\cong \prod_{[H]\in \ConjG}{(N_{G}(H)/H) \wr Sym(G/\Hbox)}.$$
\end{theorem}

For a $G$-equivariant function $\tau\in \EndH$, we define and denote its extension to $\EndG$ as
$$\widehat{\tau}(z):=\left\{\begin{array}{cc}\tau(z)& z\in \Hbox,\\ z& \text{otherwise.} \end{array} \right.$$

\begin{corollary}\label{cor:cociente-sym}
Given a group $G$ acting on a set $X$, it holds that the group 
$$\prod_{[H]\in \ConjX}{\Sym\left(G/\Hbox\right)}$$
is a quotient of $\AutG$.
\end{corollary}

\begin{proof}
To prove this assertion it is enough to find a surjective homomorphism from $\AutG$ to this product. We define a homomorphism 
$$\Pi:\AutG \rightarrow \prod_{[H]\in \ConjX}{\Sym\left(G/\Hbox\right)}$$ 
such that 
$$\Pi(\sigma)=\left(\pi_{[H]}(\sigma)\right)_{[H]\in \ConjX},$$
where $\pi_{[H]}(\sigma): G/\Hbox \rightarrow G/\Hbox$ is defined as
$$\pi_{[H]}(\sigma)(Gx):=G\sigma(x).$$

This map is well-defined, since if $Gx=Gy$, then there exists $g\in G$ such that $y=g\cdot x$. By the $G$-equivariance of $\sigma$, we have $\sigma(y)=\sigma(g\cdot x)=g\cdot\sigma(x)$, hence $G\sigma(y)=G\sigma(x)$. Thus, the image of the orbit does not depend on the chosen representative. Moreover, $\sigma$ is bijective, so $\pi_{[H]}(\sigma)$ is a permutation of the set of orbits. It is immediate that $\pi_{[H]}$ is a group homomorphism, since $\pi_{[H]}(\sigma\circ\tau)=\pi_{[H]}(\sigma)\circ\pi_{[H]}(\tau)$. Therefore, $\Pi$ is a homomorphism.

It remains to prove that $\Pi$ is surjective. Let $(\sigma_{[H]})_{[H]\in \ConjX}$ be an arbitrary element of the product. For each $[H]\in \ConjX$, the permutation $\sigma_{[H]}\in \Sym(G/\Hbox)$ induces a bijection on the set of orbits of $\Hbox$. This bijection can be extended to a $G$-equivariant function $\widetilde{\sigma_{[H]}}:\Hbox\to\Hbox$ by choosing, for each orbit, a representative and defining $\widetilde{\sigma_{[H]}}(g\cdot x_j):=g\cdot x_{\sigma_{[H]}(j)}$ (where $\sigma_{[H]}(j)$ is the index of the image orbit). This function is bijective and $G$-equivariant, hence $\widetilde{\sigma_{[H]}}\in \AutH$. Therefore, its extension $\widehat{\widetilde{\sigma_{[H]}}}\in \AutG$ is well-defined.

Now, the extensions $\widehat{\widetilde{\sigma_{[H]}}}$ act on disjoint boxes and hence commute. We define
$$\sigma := \prod_{[H]\in \ConjX} \widehat{\widetilde{\sigma_{[H]}}}.$$
This composition is well-defined pointwise: given any $x\in X$, there is a unique class $[H]\in\ConjX$ such that $x\in\Hbox$; for all $[K]\neq[H]$, the extension $\widehat{\widetilde{\sigma_{[K]}}}$ acts as the identity on $x$. Therefore, $\sigma(x)$ is determined by the unique factor acting on $x$. Clearly $\sigma\in\AutG$, since on each box $\Hbox$ it acts as the $G$-equivariant bijection $\widetilde{\sigma_{[H]}}$, and outside it as the identity.
Finally, for each $[H]\in\ConjX$, we have $\pi_{[H]}(\sigma)=\sigma_{[H]}$, since $\pi_{[H]}(\sigma)(Gx) = G\sigma(x) = \sigma_{[H]}(Gx)$ for all $x\in\Hbox$. Hence $\Pi(\sigma)=(\sigma_{[H]})_{[H]}$, proving that $\Pi$ is surjective.
\end{proof}

\section{Generating sets for the group of units $\AutG$}
In this section we find bounds for the rank of the group of units of the monoid $\EndG$, namely $\AutG$. We achieve these bounds by emphasizing special features of $G$-equivariant functions and constructing a generating set for said group. \\

We introduce and describe the nomenclature we will use to state the main theorem of this section. 
\[ \kappa_G(X) =  \vert \{ [H]\in \ConjX :|G/\Hbox|=1 \} \vert, \]
\[ \omega_G(X) =  \vert \{ [H]\in \ConjX :N_{G}(H)/H \cong1 \} \vert, \] 
$$\delta_{G}(X)=\left\{\begin{array}{cc} 0 & \omega_{G}(X)=\sum_{[H]\in \ConjX}{rank(N_{G}(H)/H)},\\ 1 & \text{otherwise.} \end{array} \right. $$
$$\Lambda_{G}(X)=\left\{\begin{array}{cc} 
1& |\ConjX|=\kappa_{G}(X)\ \text{and}\ N_{G}(H)/H\cong 1,\forall [H]\in \ConjX,\\
1 & \exists[H]\in \ConjX,\ |G/\Hbox|\geq 3,\\ 
0& \text{otherwise.} \end{array} \right. $$

We have that $\kappa_{G}(X)$ represents the number of boxes with a single orbit, $\omega_{G}(X)$ represents the number of boxes for which the corresponding quotient $N_{G}(H)/H$ is trivial, $\delta_{G}(X)$ equals one when $\omega_{G}(X)$ is equal to the sum of the ranks of all possible quotients, and $\Lambda_{G}(X)$ takes the value one for the case where $\AutG$ is a trivial group or there is at least one box with at least three orbits. In this work, following the standard convention in semigroup theory and finite generating sets, we consider the rank of the trivial group to be equal to 1.\\

We now state the main result of this section.

\begin{theorem}
Let $X$ be a finite $G$-set. Then it holds that 
\[
|\ConjX|-\kappa_{G}(X)+\Lambda_{G}(X) \] \[ \leq \rank(\AutG) \leq \] \[
|\ConjX|-\kappa_{G}(X)+\Lambda_{G}(X)-\omega_{G}(X)-\delta_{G}(X)+\sum_{[H]\in\ConjX}{\rank(N_{G}(H)/H)}.
\]
\end{theorem}

The proof of this property follows from all the properties and constructions developed below in this section.\\

\begin{proposition}\label{prop:rango-corona}
For every group $H$ and every $n\in \mathbb{N}$, the canonical projection 
\[
H \wr S_n \longrightarrow S_n,
\qquad (h_1,\dots,h_n,\sigma)\mapsto \sigma
\]
is an epimorphism. In particular,
\[
\rank(S_n) \le \rank(H \wr S_n).
\]
\end{proposition}
The projection is clearly surjective; therefore any generating set of $H \wr S_n$ projects to a generating set of $S_n$, so the minimum number of generators of the larger group is at least the minimum number of generators of $S_n$.

\begin{lemma}\label{lem:rango-simetricos}
Let $I$ be a finite set of indices, and for each $i\in I$, let $\alpha_i\ge 2$ be a natural number. Then
\[
\rank\left( \prod_{i\in I} S_{\alpha_i} \right) \ge |I| + \Lambda,
\]
where $\Lambda = 1$ if there exists $i\in I$ with $\alpha_i\ge 3$, and $\Lambda=0$ otherwise (i.e., if all $\alpha_i=2$).
\end{lemma}

\begin{proof}
Each factor $S_{\alpha_i}$ has a surjective homomorphism to $C_2$ given by the sign of the permutation. Therefore, the direct product projects onto $(C_2)^{|I|}$, which is a vector space of dimension $|I|$ over $\mathbb{F}_2$; at least $|I|$ elements are needed to generate it. Hence the direct product needs at least $|I|$ generators.

If some $\alpha_i\ge 3$, then the factor $S_{\alpha_i}$ is not cyclic (since $S_n$ is cyclic only for $n=1,2$). Therefore, in that factor at least one additional generator is needed that is not a transposition (for example, a cycle of length $\alpha_i$). This additional generator can be the same for all factors with $\alpha_i\ge 3$: simply take the product of the corresponding cycles in each factor (and the identity elsewhere). Then the rank is at least $|I|+1$ in that case.
\end{proof}

For each $i\in [r]$ with $\alpha_i\ge 2$, we define
\[
\pi_i : \AutG \longrightarrow \Sym(\alpha_i),
\qquad 
\pi_i(\rho)(j) := k \quad \text{if } \rho(Gx_j^i)=Gx_k^i.
\]
This map is a group homomorphism: if $\rho_1,\rho_2\in \AutG$, then $\pi_i(\rho_1\rho_2)=\pi_i(\rho_1)\pi_i(\rho_2)$ because the composition of permutations of orbits corresponds to the composition of the automorphisms. Moreover, $\pi_i$ is surjective: given any $\sigma\in\Sym(\alpha_i)$, the function $\rho_\sigma$ defined by $\rho_\sigma(g\cdot x_j^i):= g\cdot x_{\sigma(j)}^i$ (and the identity outside $\mathcal B_i$) is a $G$-equivariant automorphism satisfying $\pi_i(\rho_\sigma)=\sigma$.

Therefore, $\AutG$ projects onto $\prod_{i\in I} \Sym(\alpha_i)$, where $I=\{i\in[r]: \alpha_i\ge 2\}$.

\paragraph*{Proof of the lower bound.}
Let $I := \{ i\in [r] : \alpha_i \ge 2 \}$ be the set of indices of boxes with at least two orbits. Then $|I| = |\ConjX| - \kappa_G(X)$.

By Theorem 4, we have
\[
\AutG \cong \prod_{i=1}^r \left( (N_G(H_i)/H_i) \wr S_{\alpha_i} \right).
\]
For each $i\in I$, the canonical projection of the wreath product to its symmetric part (Proposition \ref{prop:rango-corona}) induces an epimorphism
\[
\AutG \twoheadrightarrow \prod_{i\in I} S_{\alpha_i}.
\]
Indeed, we can permute the orbits inside each box $\mathcal B_i$ independently, and these permutations are realized by $G$-equivariant automorphisms.

By Lemma \ref{lem:rango-simetricos}, the product $\prod_{i\in I} S_{\alpha_i}$ has rank at least
\[
|I| + \Lambda_G(X),
\]
where $\Lambda_G(X)=1$ if there is a box with $\alpha_i\ge 3$, and $0$ otherwise. 

Since the rank of a group is at least the rank of any of its quotients, we conclude that
\[
\rank(\AutG) \ge |I| + \Lambda_G(X)
= |\ConjX| - \kappa_G(X) + \Lambda_G(X).
\]

This proves the lower bound of Theorem 5.

Given a finite $G$-set $X$, for each $H_{i}\in \ConjX$ we consider a set of orbit representatives, which we denote by $X_{i}$. Without loss of generality, we can choose elements with the same stabilizer, precisely $H_{i}$. Since the box $\ibox$ has $\alpha_{i}$ orbits, in particular we have 
$$X_{i}:=\{x_{1}^{i},x_{2}^{i},...,x_{\alpha_{i}}^{i}\}.$$
Moreover, for each $i\in[r]$, $N_{G}(H_{i})/H_{i}$ is finite, and we take a generating set of minimal cardinality for it, say 
$S_{i}:=\{n_{1}^{i}H_{i},n_{2}^{i}H_{i},...,n_{k_{i}}^{i}H_{i}\}$. In particular, we denote a set of representatives of these classes as $[N_{G}(H_{i})/H_{i}]:=\{n_{1}^{i},n_{2}^{i},...,n_{k_{i}}^{i}\}$. With these elements we define and denote functions $$\eta_{t}^{i}:=(x_{i}\rightarrow n_{t}^{i}\cdot x_{i}),\mbox{ with }t\in[k_{i}].$$
Since $\{n_{1}^{i}H_{i},n_{2}^{i}H_{i},...,n_{k_{i}}^{i}H_{i}\}$ is a generating set for $N_{G}(H_{i})/H_{i}$, there exist classes $n_{1}H_{i},n_{2}H_{i},...,n_{k}H_{i}\in S_{i}$ such that 
$$nH_{i}=(n_{1}H_{i})(n_{2}H_{i})...(n_{k}H_{i})=(n_{1}...n_{k})H_{i}.$$
This implies that $n^{-1}n_{1}n_{2}...n_{k}\in H_{i}$ and
$$n^{-1}n_{1}n_{2}...n_{k}\cdot x= x,$$
$$n_{1}n_{2}...n_{k}\cdot x= n\cdot x.$$

\begin{lemma}
For $nH_{i}\in N_{G}(H_{i})/H_{i}$, there exist $n_{1},n_{2},...,n_{k}\in [N_{G}(H_{i})/H_{i}]$ such that, for any $x\in \mathcal{B}_{H_{i}}$,
$(x\rightarrow n\cdot x)= (x\rightarrow n_{k}\cdot x)...(x\rightarrow n_{2}\cdot x)(x\rightarrow n_{1}\cdot x)$.
\end{lemma}

\begin{proof}
Since the functions $(x\rightarrow g\cdot x)$, when well-defined, only map elements of the orbit of $x$ and are $G$-equivariant, it suffices to check that they map $x$ to the same element. However
$$(x\rightarrow n\cdot x)(x)=n\cdot x$$  $$ = n_{1}n_{2}...n_{k}\cdot x$$
$$=(x\rightarrow n_{k}\cdot x)(n_1n_2...n_{k-1}\cdot x)$$
$$=(x\rightarrow n_{k}\cdot x)...(x\rightarrow n_{2}\cdot x)(x\rightarrow n_{1}\cdot x)(x).$$
\end{proof}

For an orbit $Gx$, we denote by $[Gx]$ a set of orbit representatives, i.e., $g\cdot x \neq h\cdot x$ if $g,h\in [Gx]$. Note that the same set of representatives works for every orbit in the same box; hence for each $i\in [r]$ we fix and denote a set of orbit representatives simply as $[X_{i}]$. We emphasize not to confuse the set of orbit representatives inside the box $X_{i}$, which are elements of $X$, and the set of representatives inside the orbits $[X_{i}]$, which consists of elements of the group $G$. \\

Now, for any bijective $G$-equivariant function $\sigma\in \AutG$, if the $n$-cycle $(t_{1}\ t_{2}\ ...\ t_{n})$ is part of its decomposition, then every element $t_{j}$ belongs to the same box $\ibox$, and moreover the product
$\prod_{g\in [X_{i}]}(g\cdot t_{1}\ g\cdot t_{2}\ ... \ g\cdot t_{n})$ factors $\sigma$. Abusing notation, we call the product $\prod_{g\in [X_{i}]}(g\cdot t_{1}\ g\cdot t_{2}\ ... \ g\cdot t_{n})$ an $n$-\emph{orbital cycle} and denote it simply as $(t_{1}\ t_{2}\ ...\ t_{n})$. We note that for each orbital cycle, all its representatives have exactly the same stabilizer, and consequently an orbital cycle fixes all elements outside a given box. For each box $\ibox$ we then define a pair of bijective $G$-equivariant functions:
$$s_{i}=(x_{1}^{i}\ x_{2}^{i}),$$
$$t_{i}=(x_{1}^{i}\ x_{2}^{i}\ ...\ x^{i}_{\alpha_{i}}).$$

If a cycle is closed in the set of orbit representatives $X_{i}$, we call it \emph{essential}, and we say that an orbital 2-cycle is \emph{consecutive} if it is of the form $(x_{j}^{i}\ x_{j+1}^{i})$.
Since orbital cycles are composed of cycles in $\Sym(X)$ and are disjoint, they satisfy many properties analogous to ordinary cycles, although they also have their own properties which help to prove some of the outstanding results of this work. 

\begin{lemma}\label{chido}
For a finite set $X$ on which $G$ acts, the following properties hold. 
\begin{enumerate}
\item Every $n$-orbital cycle decomposes as $n-1$ orbital 2-cycles. 
\item Every essential 2-cycle decomposes into consecutive orbital 2-cycles.

\item Every consecutive orbital 2-cycle decomposes using the corresponding $t_{i}$ and $s_{i}$. 
\item Every non-essential orbital 2-cycle has a representative such that one of its elements belongs to the set of representatives of the box. 
\item Every non-essential orbital 2-cycle with representatives in a box $\Hbox$ belongs to the subgroup generated by $\{\eta_{k},t,s\}$.
\item The elements $s_i=(x_1^i\ x_2^i)$ and $t_i=(x_1^i\ x_2^i\ \cdots\ x_{\alpha_i}^i)$ generate the group of all permutations of the set of orbits of the box $\ibox$. That is,
\[
\langle s_i, t_i \rangle \cong \Sym(\alpha_i),
\]
where $\alpha_i = |G/\ibox|$.
\end{enumerate}
\end{lemma}
\begin{proof}\noindent
\begin{enumerate}
\item We have that $$(t_{1}\ t_{2}\ ...\ t_{n})=(t_{1}\ t_{2})(t_{2}\ t_{3})...(t_{k-1\ t_{k}}).$$

\item Assume without loss of generality that $n< m$, then it holds that 
$$(x^{i}_{n}\ x^{i}_{m})=(x^{i}_{n}\ x^{i}_{n+1})(x^{i}_{n+1}\ x^{i}_{n+2})...(x^{i}_{m-1}\ x^{i}_{m})(x^{i}_{m-2}\ x^{i}_{m-1})(x^{i}_{m-3}\ x^{i}_{m-2})...(x^{i}_{n}\ x^{i}_{n+1}).$$
\item It is not difficult to verify the following equality. 
$$(x^{i}_{n}\ x^{i}_{n+1})=(t_{i})^{n}(s_{i})(t_{i})^{-n}.$$
\item If $(x\ y)$ is a non-essential orbital 2-cycle, there exists $g\in G$ such that $g\cdot x\in X_{i}$ for some $i\in [r]$, and by $G$-equivariance $(x\ y)=(g\cdot x\ g\cdot y)$.
\item By the previous point, for $(x\ y)$ a non-essential orbital 2-cycle, we may assume without loss of generality that $x\in X_{H}$. Then $G_{y}=H$ and there exists $n\in N_{H}$ such that $t=n^{-1}\cdot y\in X_{H}$. Thus we can rewrite the orbital 2-cycle as $(x\ n\cdot t)$. Then the following equality holds. 
$$(x\ n\cdot t)=(n^{-1}\cdot x\ t)=(x\rightarrow n^{-1}\cdot x)(x\ t)(x\rightarrow n\cdot x).$$
Where $(x\ t)$ is an essential orbital 2-cycle.
\item Under the identification $x_j^i \leftrightarrow Gx_j^i$, the element $s_i$ is the transposition $(1\ 2)$ and $t_i$ is the cycle $(1\ 2\ \cdots\ \alpha_i)$. It is well known that $\Sym(\alpha_i)$ is generated by these two permutations; hence $\langle s_i,t_i\rangle$ acts as the full symmetric group on the orbits of $\ibox$.
\end{enumerate}
\end{proof}

A direct consequence of this lemma is the following. 
\begin{corollary}
Every $n$-orbital cycle belongs to the subgroup generated by the set 
$$\left\{(x^{i}\rightarrow n_{k_{i}}^{i}\cdot x^{i}),t_{i},s_{i}|\ i\in[r],k_{i}\in [r_{i}]\}\right.$$
\end{corollary}

\begin{claim}
For any orbit representatives $p,q,x$ of a box $\Hbox$ and any $n,m\in [N_{G}(H)/H]$, it holds that 
$$(x\ p)(x\rightarrow n\cdot x)(x\ p)(x\ q)(x\rightarrow m\cdot x)(x\ q)=(x\ q)(x\rightarrow m\cdot x)(x\ q)(x\ p)(x\rightarrow n\cdot x)(x\ p).$$
\end{claim}
That is, the elements of the form $(x\ p)(x\rightarrow n\cdot x)(x\ p)$ commute with each other. This claim is easy to verify by carefully performing the corresponding evaluations.

\begin{lemma}
For each $G$-equivariant function $\tau\in \AutG$, for each $i\in[r]$, there exist functions defined as 
$$\tau_{i}(z):=\left\{\begin{array}{cc} \tau(z) & z\in \ibox, \\ x & \text{otherwise,} \end{array} \right.$$
such that $$\tau=\tau_{1}\tau_{2}...\tau_{r}.$$
Moreover, these commute with each other. 
\end{lemma}

\begin{proof}
For any $z\in \ibox$ and any $i\in[r]$, we have $\tau(x)\in \ibox$. Consequently $\tau_{k}(z)=z$ and $\tau_{k}(\tau(z))=\tau(z)$ for all $k\neq i$. As a consequence of these observations, the equality holds for any order of the functions $\tau_{i}$, yielding $\tau(z)$. 
\end{proof}

\begin{theorem}
Given a finite set $X$ and a group $G$ acting on it, for a set of representatives of a minimal generating set of $N_{G}(H_{i})/H_{i}$, say $\{n_{1}^{i},n_{2}^{i},...,n_{r_{i}}^{i}\}$, and sets of orbit representatives in the boxes $X_{i}$, it holds that$$\AutG=\langle \{(x^{i}\rightarrow n^{i}_{k_{i}}\cdot x^{i}),t_{i},s_{i}|\ i\in[r],\ k_{i}\in[r_{i}]\} \rangle.$$
\end{theorem}

\begin{proof}
Let $\tau\in \AutG$, and consider its decomposition $\tau=\tau_{1}\tau_{2}...\tau_{r}$. It is enough to prove that $\tau_{i}$ is generated by the proposed set for any $i\in [r]$. We propose a new decomposition for these functions. Consider the sets 
$$\ibox':=\{z\in \ibox|\ \tau(Gz)=Gz\},$$
$$\ibox'':=\{z\in \ibox|\ \tau(Gz)\neq Gz\}.$$
Then we define a pair of functions as follows:
$$\tau_{i}'(z):=\left\{ \begin{array}{cc}\tau_{i}(z)& z\in \ibox'\\ z& \text{otherwise.}  \end{array}\right.$$
$$\tau_{i}''(z):=\left\{ \begin{array}{cc}\tau_{i}(z)& z\in \ibox''\\ z& \text{otherwise.}  \end{array}\right.$$
It is not hard to see that $\tau_{i}=\tau'_{i}\tau''_{i}$. 

Note that each set $\ibox'$ and $\ibox''$ is $G$-invariant, closed under the group action, and for each element $x^{i}_{t}\in \ibox'\cap X_{i}$ there exists $n_{t}\in N_{G}(H_{i})/H_{i}$ such that 
$$\tau(x^{i}_{t})=\tau_{i}(x^{i}_{t})=\tau'_{i}(x^{i}_{t})=n_{t}\cdot x^{i}_{t}.$$
Denoting by $[\ibox'\cap X_{i}]$ the set of indices of the elements $x^{i}_{t}$ in $\ibox'\cap X_{i}$, some calculations show that 
$$\tau'_{i}=\prod_{t\in [\ibox'\cap X_{i}]}(x^{i}\rightarrow n_{t}\cdot x^{i}).$$
From Lemma \ref{chido} we know that each $(x^{i}\rightarrow n_{t}\cdot x^{i})$ is generated by the set $\{\eta_{k_{i}}^{i}: \ k_{i}\in [r_{i}]\}$. Consequently
$$\tau'_{i}\in \langle\{\eta_{k_{i}}^{i}: \ k_{i}\in [r_{i}] \}\rangle \leq \langle\eta_{k_{i}}^{i},t_{i},s_{i}:i\in[r], \ k_{i}\in [r_{i}] \rangle.$$

On the other hand, for each element $z\in \mathcal B_i''$, consider the $G$-orbit $Gz$. Since $\mathcal B_i''$ is finite (as $X$ is finite), the function $\tau_i''$ induces a permutation on the set of $G$-orbits contained in $\mathcal B_i''$. Therefore, there exists a minimal natural number $m$ such that $(\tau_i'')^m(Gz)=Gz$. Moreover, since $(\tau_i'')^m(z)\in Gz$, there exists $n\in G$ such that $(\tau_i'')^m(z)=n\cdot z$. Since $(\tau_i'')^m$ is $G$-equivariant and preserves the box $\mathcal B_i$, we have $G_{(\tau_i'')^m(z)}=G_z$; consequently, $n^{-1}G_z n = G_z$, i.e., $n\in N_G(H_i)$. Hence $n$ can be considered as an element of $N_G(H_i)/H_i$.

We may assume without loss of generality that $z=x_t^i$ is the representative of its orbit in the set $X_i$. The restriction of $\tau_i''$ to the union of orbits
\[
Gx_t^i \cup G\tau(x_t^i) \cup \cdots \cup G\tau^{m-1}(x_t^i)
\]
decomposes as
\[
\tau_i''|_{Gx_t^i \cup \cdots \cup G\tau^{m-1}(x_t^i)}
=
(x^i\ x_t^i)(x^i\to n\cdot x^i)(x^i\ x_t^i)
\, (x_t^i\ \tau(x_t^i)\ \tau^2(x_t^i)\ \cdots \ \tau^{m-1}(x_t^i)).
\]
Indeed, the last factor is the $m$-orbital cycle permuting the representatives $x_t^i, \tau(x_t^i), \dots, \tau^{m-1}(x_t^i)$. The factor
\[
(x^i\ x_t^i)(x^i\to n\cdot x^i)(x^i\ x_t^i)=(x_t^i\to n\cdot x_t^i)
\]
acts as the internal translation on the orbit $Gx_t^i$ once the orbital cycle returns to the starting point. The verification is carried out by evaluating directly on elements of the form $g\cdot \tau^q(x_t^i)$, with $q=0,1,\dots,m-1$ and $g\in [Gx_t^i]$.

Varying $z$ over a set of representatives of the orbits under this permutation, we obtain a partition of $\mathcal B_i''$ into disjoint cycles. Let $[\mathcal B_i''/\sim]$ denote the set of indices of the representatives $x_t^i$ that start each cycle. Then
\[
\tau_i''=
\prod_{t\in [\mathcal B_i''/\sim]}
(x^i\ x_t^i)(x^i\to n_t\cdot x^i)(x^i\ x_t^i)
\, (x_t^i\ \tau(x_t^i)\ \tau^2(x_t^i)\ \cdots \ \tau^{m_t-1}(x_t^i)),
\]
where $m_t$ is the minimal natural number and $n_t\in N_G(H_i)/H_i$ is the corresponding element for the cycle beginning at $x_t^i$.

Each factor of this product is an $m_t$-orbital cycle (the right-hand factor) conjugated by elements of the proposed set; by Lemma \ref{chido}, such a factor belongs to the subgroup generated by $\{(x^i\to n_{k_i}^i\cdot x^i), t_i, s_i\}$. Therefore, $\tau_i''$ belongs to that subgroup. Since $\tau_i'$ also belongs to it (by the second part) and $\tau_i=\tau_i'\tau_i''$, it follows that $\tau_i$ lies in the subgroup. Hence $\tau=\tau_1\cdots\tau_r$ also lies in it, completing the proof.

\end{proof}

Next, we will use an elementary property of generating sets of a group to reduce the number of generators. Recall that if $V$ is a subset of a group $G$ and $t,s\in G$, then
\[
\langle V,t,s\rangle = \langle V,ts\rangle \quad\Longleftrightarrow\quad t,s\in \langle V,ts\rangle.
\]
This equivalence will allow us to eliminate some redundant generators by combining several into a single element, as long as the remaining generators allow us to recover the factors.

\begin{theorem}
Given a finite set $X$ and a group $G$ acting on it, it holds that 
$$\AutG=\langle \{(x\rightarrow n^{j}_{i_{j}}\cdot x),\prod_{k=1}^{r}t_{k},s_{j}|\ j\in[r],\ i_{j}\in[r_{j}]\} \rangle.$$
\end{theorem}

\begin{proof}
Let $B$ be the set of all generators of the form $(x\rightarrow n^{j}_{i_j}\cdot x)$, and let
$T := \prod_{k=1}^{r} t_k$.
Define
$H := \langle B \cup \{T\} \cup \{s_1,\dots,s_r\} \rangle$.
By Theorem 12 we know that
\[
\AutG = \langle B \cup \{t_1,\dots,t_r\} \cup \{s_1,\dots,s_r\} \rangle.
\]
Therefore, it suffices to prove that $t_j \in H$ for each $j\in[r]$.

Fix $j\in[r]$. Since $T$ commutes with all $t_k$ for $k\neq j$, for any $k\in \mathbb{Z}$ we have
\[
T^k s_j T^{-k} = t_j^k s_j t_j^{-k}.
\]
If $s_j = (x_1^j\ x_2^j)$ and $t_j = (x_1^j\ x_2^j\ \cdots\ x_{\alpha_j}^j)$, then
\[
t_j^k s_j t_j^{-k} = (x_{k+1}^j\ x_{k+2}^j)
\]
(where indices are taken modulo $\alpha_j$). In particular, for $k=0,1,\dots,\alpha_j-2$, these elements are the consecutive transpositions
\[
(x_1^j\ x_2^j),\ (x_2^j\ x_3^j),\ \dots,\ (x_{\alpha_j-1}^j\ x_{\alpha_j}^j).
\]
The product of these consecutive transpositions is exactly the cycle
\[
(x_1^j\ x_2^j\ \cdots\ x_{\alpha_j}^j) = t_j.
\]
Therefore,
\[
t_j = \prod_{k=0}^{\alpha_j-2} T^k s_j T^{-k}.
\]
Since $T, s_j \in H$, each factor $T^k s_j T^{-k}$ belongs to $H$, and consequently $t_j\in H$.

This shows that all $t_j$ lie in $H$. Since $B\subseteq H$ and $s_j\in H$, we have
\[
\AutG = \langle B \cup \{t_1,\dots,t_r\} \cup \{s_1,\dots,s_r\} \rangle \subseteq H.
\]
The inclusion $H\subseteq \AutG$ is immediate by definition. Hence $H=\AutG$, which is exactly the desired equality.
\end{proof}

\begin{theorem}\label{rank}
Given a finite set $X$ and a group $G$ acting on it, it holds that 
$$\AutG=\langle \{(x\rightarrow n^{j}_{i_{j}}\cdot x),s_{j},(x_{1}\rightarrow n_{1}^{1}\cdot x_{1})\prod_{k=1}^{r}t_{k}|\ j\in[r],\ i_{j}\in[\alpha_{j}]\}\setminus \{(x_{1}\rightarrow n_{1}^{1}\cdot x_{1})\} \rangle.$$
\end{theorem}

\begin{proof}
Let $B$ be the set of all generators of the form $(x\rightarrow n^{j}_{i_j}\cdot x)$, and let
$\eta_1^1 := (x_1 \to n_1^1 \cdot x_1)$.
Let $T := \prod_{k=1}^{r} t_k$ and define
\[
H := \langle \{B\setminus\{\eta_1^1\}\} \cup \{ \eta_1^1 T \} \cup \{s_1,\dots,s_r\} \rangle.
\]
By Theorem 13 we know that
\[
\AutG = \langle B \cup \{T\} \cup \{s_1,\dots,s_r\} \rangle.
\]
It suffices to prove that $\eta_1^1 \in H$ and $T \in H$ to conclude that $H=\AutG$.

To this end, write $T = t_1 U$, where $U = \prod_{k=2}^r t_k$. Note that $\eta_1^1$ and $s_1$ commute with $U$. Let
\[
\eta_1^{1,(2)} := s_1 \eta_1^1 s_1^{-1},
\]
which belongs to $B\setminus\{\eta_1^1\}\subset H$. Then, from the definition of $H$, we have $\eta_1^1 T \in H$ and $\eta_1^{1,(2)} \in H$.

Observe that
\[
\eta_1^1 T = \eta_1^1 t_1 U.
\]
On the other hand,
\[
s_1 (\eta_1^1 T) s_1^{-1} = \eta_1^{1,(2)} (s_1 t_1 s_1^{-1}) U.
\]
Multiplying the first expression by the inverse of the second (and using that $U$ commutes with $\eta_1^1$ and $s_1$), we obtain:
\[
(\eta_1^1 T) \bigl( s_1 (\eta_1^1 T) s_1^{-1} \bigr)^{-1}
= \eta_1^1 t_1 (s_1 t_1^{-1} s_1^{-1})
= \eta_1^1 c,
\]
where $c := t_1 s_1 t_1^{-1} s_1^{-1}$.

Now, $c$ is a permutation acting only on the box $\mathcal B_1$, and a direct calculation shows that
\[
c = t_1 s_1 t_1^{-1} s_1^{-1} = (x_1^1\ x_3^1\ x_2^1)
\]
(if $\alpha_1=1$, then $c=\mathrm{id}$; if $\alpha_1=2$, then $c=s_1$). Therefore, $c$ is generated by $s_1$ and $t_1$, and by Theorem 13 (or more concretely, by the decomposition of $t_1$ in terms of $T$ and $s_1$), we have $c \in H$.

Thus $\eta_1^1 c \in H$ and $c \in H$, which implies
\[
\eta_1^1 = (\eta_1^1 c) c^{-1} \in H.
\]

Now, from $\eta_1^1 T \in H$ and $\eta_1^1 \in H$, we get
\[
T = (\eta_1^1)^{-1} (\eta_1^1 T) \in H.
\]

Therefore, $H$ contains $B\setminus\{\eta_1^1\}$, $\eta_1^1$, $T$, and all $s_j$. By Theorem 13, this implies $\AutG \subseteq H$. The reverse inclusion is trivial, hence $H=\AutG$.
\end{proof}

Based on all the constructions above, the upper bound of the theorem is obtained by counting the elements of the generating set constructed in Theorem \ref{rank}. 
This set consists of:
\begin{enumerate}
\item The base generators \((x\rightarrow n^{j}_{i_j}\cdot x)\), except one, \((x_1\rightarrow n_1^1\cdot x_1)\), which has been merged into a single element with the product of the cycles.
\item The product \(T=\prod_{k=1}^r t_k\), which replaces all generators \(t_i\).
\item The transpositions \(s_j\), which exist only when the corresponding box has at least two orbits.
\end{enumerate}
Counting now:
Initially there are \(\sum_{[H]\in\ConjX} \rank(N_G(H)/H)\) base generators. Among them, those corresponding to trivial quotients (\(N_G(H)/H\cong 1\)) are the identity and are omitted, which subtracts \(\omega_G(X)\). Furthermore, Theorem 14 removes an additional generator \((x_1\to n_1^1\cdot x_1)\) by merging it with \(T\); this adjustment, along with the special cases where this generator is already redundant, gives rise to the terms \(\Lambda_G(X)\) and \(\delta_G(X)\) in the formula. 
On the other hand, the number of transpositions \(s_j\) is \(|\ConjX|-\kappa_G(X)\), since there is one for each box with at least two orbits. Finally, the contribution of the orbital cycles reduces to a single generator, \(T=\prod_{k=1}^r t_k\), instead of the original \(r\) generators \(t_i\). 
Thus we obtain the upper bound
\[
|\ConjX|-\kappa_G(X)+\Lambda_G(X)-\omega_G(X)-\delta_G(X)+\sum_{[H]\in\ConjX}\rank(N_G(H)/H).
\]

\subsection{Improving the upper bound}

In this section we present a result and an algorithm that allow reducing the number of generators of the base 
$\{\eta^{i}_{k_{i}} \mid i\in[r],\, k_i\in[r_i]\}$, 
whose cardinality is $\sum_{[H]\in\ConjX} \rank(N_G(H)/H)$, 
by grouping those whose orders are pairwise coprime.

\begin{lemma}\label{pegado}
Let $\eta^{i}_{p}$ and $\eta^{j}_{q}$ be two base generators with $i\neq j$, whose orders $m:=o(\eta^{i}_{p})$ and $n:=o(\eta^{j}_{q})$ are coprime. Then
\[
\langle \eta^{i}_{p},\, \eta^{j}_{q} \rangle
=
\langle \eta^{i}_{p}\eta^{j}_{q} \rangle.
\]
\end{lemma}
\begin{proof}
Since $i\neq j$, the generators $\eta^{i}_{p}$ and $\eta^{j}_{q}$ act on distinct boxes, hence commute. Let $P := \eta^{i}_{p}\eta^{j}_{q}$. Clearly $P$ belongs to the subgroup generated by $\eta^{i}_{p}$ and $\eta^{j}_{q}$, so $\langle P \rangle \subseteq \langle \eta^{i}_{p}, \eta^{j}_{q} \rangle$.

For the reverse inclusion, we will show that both $\eta^{i}_{p}$ and $\eta^{j}_{q}$ are powers of $P$. Since $\gcd(m,n)=1$, there exist integers $u$ and $v$ such that:
\[
u \equiv 1 \pmod{m}, \qquad u \equiv 0 \pmod{n},
\]
\[
v \equiv 0 \pmod{m}, \qquad v \equiv 1 \pmod{n}.
\]
(These exist by the Chinese remainder theorem; explicitly, if $am+bn=1$, we may take $u = b n$ and $v = a m$.)

Then, using that $\eta^{i}_{p}$ and $\eta^{j}_{q}$ commute, we have
\[
P^u = (\eta^{i}_{p}\eta^{j}_{q})^u
= (\eta^{i}_{p})^u (\eta^{j}_{q})^u
= \eta^{i}_{p} \cdot \mathrm{id}
= \eta^{i}_{p},
\]
since $(\eta^{i}_{p})^u = \eta^{i}_{p}$ (as $u\equiv 1 \pmod{m}$) and $(\eta^{j}_{q})^u = \mathrm{id}$ (as $u\equiv 0 \pmod{n}$).

Similarly,
\[
P^v = (\eta^{i}_{p})^v (\eta^{j}_{q})^v
= \mathrm{id} \cdot \eta^{j}_{q}
= \eta^{j}_{q}.
\]
Therefore, $\eta^{i}_{p}$ and $\eta^{j}_{q}$ belong to $\langle P \rangle$, so $\langle \eta^{i}_{p}, \eta^{j}_{q} \rangle \subseteq \langle P \rangle$. Combining both inclusions gives the desired equality.
\end{proof}

\begin{corollary}\label{cor:pegado-conmutante}
Under the same hypotheses of Lemma \ref{pegado}, if $B$ is any set of generators that commute with $\eta^{i}_{p}$ and $\eta^{j}_{q}$, then
\[
\langle \{\eta^{i}_{p},\, \eta^{j}_{q}\} \cup B \rangle
=
\langle \{\eta^{i}_{p}\eta^{j}_{q}\} \cup B \rangle.
\]
In particular, this applies when $B$ consists of base generators $\eta^{l}_{k_l}$ with $l\in[r]\setminus\{i,j\}$.
\end{corollary}
This follows directly from Lemma \ref{pegado}, since all elements of $B$ commute with $\eta^{i}_{p}$, $\eta^{j}_{q}$, and their product $P$.

\begin{corollary}\label{cor:pegado-general}
Lemma \ref{pegado} extends to any finite number of generators $\eta^{i_1}_{k_{i_1}}, \eta^{i_2}_{k_{i_2}}, \dots, \eta^{i_m}_{k_{i_m}}$ with distinct $i_1,\dots,i_m$ and pairwise coprime orders. In that case,
\[
\langle \eta^{i_1}_{k_{i_1}},\, \eta^{i_2}_{k_{i_2}},\, \dots,\, \eta^{i_m}_{k_{i_m}} \rangle
=
\langle \eta^{i_1}_{k_{i_1}} \eta^{i_2}_{k_{i_2}} \cdots \eta^{i_m}_{k_{i_m}} \rangle.
\]
Moreover, if $B$ is a set of generators commuting with all $\eta^{i_t}_{k_{i_t}}$, then
\[
\langle \{\eta^{i_t}_{k_{i_t}} \mid t=1,\dots,m\} \cup B \rangle
=
\langle \{\prod_{t=1}^m \eta^{i_t}_{k_{i_t}}\} \cup B \rangle.
\]
\end{corollary}
This follows by applying Lemma \ref{pegado} iteratively: first merge $\eta^{i_1}_{k_{i_1}}$ and $\eta^{i_2}_{k_{i_2}}$, then the result with $\eta^{i_3}_{k_{i_3}}$, and so on. In each step, the order of the partial product is the product of the orders of the factors (since they are pairwise coprime), so it remains coprime to the order of the next generator. The extension to a commuting set $B$ is obtained by applying Corollary \ref{cor:pegado-conmutante} at each step.\\

\textbf{Important remark.} 
The lemma and its corollaries only apply to generators that commute with each other. In particular, the generators $t_i$ and $s_i$ do not commute with $\eta^{i}_{k_i}$ (since they permute the orbits), so they cannot be absorbed into the product of a packet. The grouping algorithm presented below only affects the base generators $\eta^{i}_{k_i}$; the generators $t_i$ and $s_i$ remain as independent elements in the final generating set.

\medskip
\noindent
\textbf{Definition of \(\lambda_G^{\min}(X)\).}
Given a finite \(G\)-set \(X\), consider the set of non-trivial base generators
\[
\mathcal{B} := \{ \eta^{i}_{k_i} \mid i\in[r],\, k_i\in[r_i],\, o(\eta^{i}_{k_i}) > 1 \}.
\]
A \emph{packet} is a subset of \(\mathcal{B}\) satisfying:
\begin{enumerate}
\item[(i)] all its elements have pairwise coprime orders;
\item[(ii)] it does not contain two generators with the same superscript \(i\).
\end{enumerate}
A \emph{partition into packets} of \(\mathcal{B}\) is a collection of disjoint packets whose union is \(\mathcal{B}\). We denote by
\[
\lambda_G^{\min}(X)
\]
the minimum number of packets in a partition of \(\mathcal{B}\) satisfying conditions (i) and (ii).

Since \(\mathcal{B}\) is finite, the minimum exists. Clearly,
\[
\lambda_G^{\min}(X) \le \sum_{[H]\in\ConjX} \rank(N_G(H)/H) - \omega_G(X),
\]
where $\omega_G(X)$ is the number of trivial generators (corresponding to quotients $N_G(H)/H\cong 1$) that were omitted from $\mathcal{B}$. Equality holds in the worst case, when no pair of non-trivial generators has coprime orders.\\

\textbf{Grouping algorithm.}\\
The procedure is as follows:

\begin{enumerate}
\item Fix an order to traverse the superscripts $i\in[r]$ (e.g., $1,2,\dots,r$). 
\item Take the first available generator $\eta^{i}_{k_i}$ (starting with $i=1$) and place it in a new packet $P_1$.
\item Traverse the remaining superscripts in the fixed order. For each superscript $j$, if there exists a generator $\eta^{j}_{t}$ that has not been used and whose order is coprime to the orders of all elements in $P_1$, add it to $P_1$; otherwise, continue to the next superscript.
\item Repeat the previous step until all superscripts have been traversed without being able to add more elements to $P_1$.
\item Once $P_1$ is complete, choose a new generator $\eta^{j}_{t'}$ that is not in any already formed packet, and repeat the process to build $P_2$.
\item Continue until all non-trivial generators have been assigned to a packet.
\end{enumerate}

This procedure produces a partition of the set of non-trivial generators into packets, each with properties (i) and (ii) from the definition above. Property (i) ensures that the product of the elements of a packet has order equal to the product of the individual orders, and thus generates the direct product of the corresponding cyclic subgroups. Property (ii) guarantees that generators in the same packet act on distinct boxes, so their direct products are compatible and do not introduce conflicts. Thus, replacing each packet by the product of its elements yields a new generating set for the base of the wreath product.\\

Replacing in the upper bound of Theorem 6 the term $\sum \rank(N_G(H)/H) - \omega_G(X) - \delta_G(X)$ by $\lambda_G^{\min}(X)$, we obtain a substantial improvement.

\begin{corollary}
Given a group $G$ acting on a finite set $X$, it holds that 
\[
\rank(\AutG) \le |\ConjX|-\kappa_G(X)+\Lambda_G(X)+\lambda_G^{\min}(X).
\]
\end{corollary}
Indeed, the new bound is always less than or equal to the previous one, since
\[
\lambda_G^{\min}(X) \le \sum_{[H]\in\ConjX} \rank(N_G(H)/H) - \omega_G(X) - \delta_G(X).
\]
The inequality is strict whenever it is possible to group at least two generators of coprime orders in the same packet.

\begin{example}
Consider a set $X$ with 21 elements such that 
$$N_{G}(H_{1})/H_{1}\cong \zz_{2}\times \zz_{2},$$
$$N_{G}(H_{2})/H_{2}\cong \zz_{3},$$ 
for which a diagram is generated as shown in Figure~\ref{figurados}.\\
\begin{figure}[ht]\label{figurados}
\centering
\includegraphics[width=2.5in]{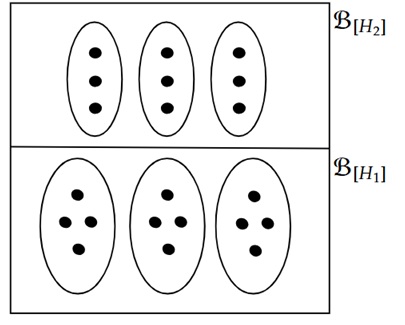}
\caption{A $G$-set}
\end{figure}
Using the results in \cite{paper3} we can count that for this particular $G$-set it holds that 
$$|\EndG| = 3^9 \cdot 7^3 = 6,751,269,$$
$$|\AutG| = 2^{8} \cdot 3^5 = 62,208.$$
Let $\{x_{1},x_{2},x_{3}\}$ and $\{y_{1},y_{2},y_{3}\}$ be the sets of orbit representatives for each box respectively, and denote 
$$\eta^{1}_{1}=(1,0),\quad \eta^{1}_{2}=(0,1),\quad \eta^{2}_{1}=1.$$
In the previous section it was shown that the set 
$$\{(y\rightarrow 1\cdot y),(x\rightarrow (1,0)\cdot x),(x\rightarrow (0,1)\cdot x),(y_{1}\ y_{2}\ y_{3}),(x_{1}\ x_{2}\ x_{3}),(x_{1}\ x_{2}),(y_{1}\ y_{2})\}$$
generates $\AutG$, implying that its rank is at most 5.

Now, the algorithm allows generating the packets
$$\{(y\rightarrow 1\cdot y),\ (x\rightarrow(1,0)\cdot x)\}\quad\text{and}\quad \{(x\rightarrow (0,1)\cdot x)\}.$$
Lemma \ref{pegado} implies that the set 
$$\{(y\rightarrow 1\cdot y)(x\rightarrow (1,0)\cdot x),\ (x\rightarrow (0,1)\cdot x),\ (y_{1}\ y_{2}\ y_{3}),\ (x_{1}\ x_{2}\ x_{3}),\ (x_{1}\ x_{2}),\ (y_{1}\ y_{2})\}$$
also generates $\AutG$, showing that its rank is at most 4. 
\end{example}

\medskip
\noindent
\textbf{Open problem.}
The algorithm presented provides a practical way to improve the upper bound, but it leaves open the fundamental question of determining the exact value of \(\lambda_G^{\min}(X)\) in terms of intrinsic properties of the action. 

In essence, the bound
\[
\rank(\AutG) \le |\ConjX|-\kappa_G(X)+\Lambda_G(X)+\lambda_G^{\min}(X)
\]
should be the exact value of the rank of \(\AutG\); however, the lack of an explicit expression for \(\lambda_G^{\min}(X)\) makes it difficult to improve the lower bound proved in the previous section and, consequently, to prove equality.

\section{Finiteness conditions in generating sets}

In this section we establish necessary conditions for the monoid $\EndG$ to be finitely generated. The following results extend the ideas of \cite{paper} to arbitrary $G$-sets.

\begin{definition}[Orbital infiltration]
Given subgroups \( H,K\leq G \) such that \( H\leq K \), we call a function \( \tau \in \EndG \) an \emph{orbital infiltration} of type \( (H,[K]_{N_{H}}) \) if it satisfies the following conditions:
\begin{enumerate}
\item \(\tau(\mathcal{B}_{[T]})\subseteq \mathcal{B}_{[T]}\), for all \([T]\in \ConjX\setminus\{[H],[K]\}\). 
\item \(\tau|_{\mathcal{B}_{[T]}}\) is an injective function from \(\mathcal{B}_{[T]}\) into itself, for all \([T]\in \ConjX\setminus\{[H],[K]\}\).
\item There exists \(x\in \Hbox\) such that \(\tau(x)\in \Kbox\).
\item If \(\tau(x)\notin \Hbox\), then \([G_{\tau(x)}]_{N_{H}}=[K]_{N_{H}}\), for all \(x\in \Hbox\).
\end{enumerate}
\end{definition}

The following theorem is a strong condition for the generating sets of $\EndG$ and is the cornerstone of this section. Its proof is an adaptation of the argument used in \cite[Theorem 17]{paper} for finite sets.

\begin{theorem}\label{thm:infiltraciones}
Given a set $X$ on which a group $G$ acts, every generating set of $\EndG$ contains an orbital infiltration of each possible type.  
\end{theorem}

\begin{proof}
Suppose, for contradiction, that there exists a type $(H,[K]_{N_H})$ such that no element of \(U\subseteq \EndG\) is an orbital infiltration of that type, but \(\EndG = \langle \AutG \cup U \rangle\). Let \(\tau\in\EndG\) be an orbital infiltration of type \((H,[K]_{N_H})\). Then \(\tau\) can be expressed as \(\tau = \rho_s \cdots \rho_1\), where each \(\rho_i \in \AutG \cup U\). Let \(t\) be the smallest index such that \(\rho_t\) is not bijective. The factors \(\rho_1,\dots,\rho_{t-1}\) are bijections, so they preserve stabilizers and boxes.

For any \([T]\neq[H],[K]\), the restriction \(\tau|_{\mathcal{B}_{[T]}}\) is injective. This forces \(\rho_t|_{\mathcal{B}_{[T]}}\) to be injective; otherwise, there would exist \(x\neq y\) in \(\mathcal{B}_{[T]}\) with \(\rho_t(x)=\rho_t(y)\), and then \(\tau(x)=\tau(y)\), contradicting the injectivity of \(\tau\). Moreover, \(\rho_t\) must preserve each \(\mathcal{B}_{[T]}\) with \([T]\neq[H],[K]\), since if there existed \(x\in\mathcal{B}_{[T]}\) such that \(\rho_t(x)\notin\mathcal{B}_{[T]}\), then, since \(G\)-equivariant functions respect the partial order of boxes (Lemma \ref{lema1}), the final image \(\tau(x)\) would not belong to \(\mathcal{B}_{[T]}\), contradicting condition (1) of the definition of orbital infiltration.

It remains to verify the conditions on the box \(\Hbox\). By condition (iii) of \(\tau\), there exists \(x_0\in\Hbox\) such that \(\tau(x_0)\in\Kbox\). Let \(x_0' = \rho_{t-1}\cdots\rho_1(x_0)\), which belongs to \(\Hbox\) since the previous factors are bijections. We claim that \(\rho_t(x_0')\in\Kbox\). If \(\rho_t(x_0')\) were in an intermediate box \([J]\) with \([H]\le [J] < [K]\), then for \(\tau(x_0)=\rho_s\cdots\rho_t(x_0')\in\Kbox\), the composition of the remaining factors \(\rho_{t+1}\cdots\rho_s\) would have to move elements from \(\mathcal{B}_{[J]}\) to \(\mathcal{B}_{[K]}\). But \(\tau\) is injective on \(\mathcal{B}_{[J]}\) (since \([J]\neq[H],[K]\)), and this injectivity is inherited from the restriction of \(\rho_t\) to the preimage of \(\mathcal{B}_{[J]}\). Consequently, \(\rho_t\) must already identify the fibers necessary for the final composition to be injective on \(\Hbox\), which forces \(\rho_t\) to satisfy condition (iv) of the definition of orbital infiltration: if \(\rho_t(x)\notin\Hbox\) for some \(x\in\Hbox\), then \([\rho_t(x)]_{N_H}=[K]_{N_H}\). Indeed, if \([\rho_t(x)]_{N_H}\) were an intermediate class \([J]_{N_H}\) with \([H]<[J]<[K]\), then by the injectivity of \(\tau\) on \(\mathcal{B}_{[J]}\), no later factor could modify that conjugacy class without creating a collision in the kernel of \(\tau\), contradicting that \(\rho_t\) is the first non-bijective factor and that \(\tau\) satisfies (iv). Hence \(\rho_t\) satisfies condition (iv). Moreover, by the previous discussion, \(\rho_t\) is injective and preserves boxes distinct from \([H]\) and \([K]\), and we have just proved that there exists \(x_0'\in\Hbox\) with \(\rho_t(x_0')\in\Kbox\). Hence \(\rho_t\) is an orbital infiltration of the same type \((H,[K]_{N_H})\). Since \(\rho_t\) is not bijective, it cannot belong to \(\AutG\), so \(\rho_t\in U\), contradicting the assumption that \(U\) contains no infiltrations of this type. We conclude that \(U\) must contain at least one orbital infiltration of each type.
\end{proof}

As a direct consequence, if there are infinitely many types of orbital infiltrations, then \(\EndG\) cannot be finitely generated.

\begin{corollary}\label{cor:orbitasinfinitas}
Let $X$ be a set such that there exists $H\in \StabG$ for which $G/\Hbox$ is infinite. Then $\AutG$ is not finitely generated, and consequently $\EndG$ is not finitely generated.
\end{corollary}

\begin{proof}
If $G/\Hbox$ is infinite, the set of orbits $\mathcal{O} := \Hbox/G$ is infinite. The symmetric group $\Sym(\mathcal{O})$ is uncountable. Fixing a representative $x_i$ for each orbit $\mathcal{O}_i\in\mathcal{O}$, every permutation $\sigma\in\Sym(\mathcal{O})$ induces a $G$-equivariant bijection $\rho_\sigma:X\to X$ defined by
\[
\rho_\sigma(g\cdot x_i):= g\cdot x_{\sigma(i)},
\]
and acting as the identity on $X\setminus\Hbox$. The assignment $\sigma\mapsto \rho_\sigma$ is a group monomorphism, so $\Sym(\mathcal{O})$ is a subgroup of $\AutG$. If $\AutG$ were finitely generated, it would be countable, but it contains an uncountable subgroup, a contradiction. Hence $\AutG$ is not finitely generated. Consequently, $\EndG$ cannot be finitely generated either, since in a finitely generated monoid the group of units is finitely generated.
\end{proof}
\begin{corollary}
Let $G$ be a group and $X$ a $G$-set such that:
\begin{enumerate}
\item[(i)] \(|\ConjX|=\infty\);
\item[(ii)] For every \([H]\in\ConjX\setminus\{[G]\}\) (except possibly finitely many), there exists \([K]\in\ConjX\) such that \([H]<[K]\).
\end{enumerate}
Then \(\EndG\) is not finitely generated.
\end{corollary}

\begin{proof}
By condition (ii), for all but finitely many classes \([H]\), there exists a strictly larger class \([K]\). For each pair \(([H],[K])\) with \(H<K\), there is at least one type of orbital infiltration \((H,[K]_{N_H})\). Since there are infinitely many pairs \(([H],[K])\) (as \(|\ConjX|=\infty\) and the chain is ascending), there are infinitely many distinct types. By Theorem \ref{thm:infiltraciones}, any generating set must contain a function for each type, hence it must be infinite. Therefore \(\EndG\) is not finitely generated.
\end{proof}

\begin{corollary}
Let $G$ be a group and $X$ a $G$-set such that:
\begin{enumerate}
\item[(i)] \(|\ConjX|=\infty\);
\item[(ii)] There exists \([H_0]\in\ConjX\) such that there are infinitely many \([K_\alpha]\in\ConjX\) with \([H_0]<[K_\alpha]\).
\end{enumerate}
Then \(\EndG\) is not finitely generated.
\end{corollary}

\begin{proof}
The infinite collection of pairs \(([H_0],[K_\alpha])\) yields infinitely many types of orbital infiltrations of the form \((H_0, [K_\alpha]_{N_{H_0}})\). By Theorem \ref{thm:infiltraciones}, infinitely many generators are needed. Hence \(\EndG\) is not finitely generated.
\end{proof}

\begin{lemma}\label{lem:tipos}
Let \(K_1, K_2 \leq H \leq G\) be subgroups of \(G\) such that \(K_1, K_2 \subsetneq H\). The orbital infiltrations of types \((K_1, [H]_{N_{K_1}})\) and \((K_2, [H]_{N_{K_2}})\) are of the same type if and only if there exists \(n \in N_G(H)\) such that
\[
K_2 = n^{-1} K_1 n.
\]
\end{lemma}

\begin{proof}
By definition, the type of an orbital infiltration is determined by the pair \((K, [H]_{N_K})\), where \(K\) is the stabilizer subgroup of the source box and \([H]_{N_K}\) is the \(N_K\)-conjugacy class of \(H\). Two types are equal if and only if their corresponding pairs are equal.

Suppose first that there exists \(n \in N_G(H)\) such that \(K_2 = n^{-1} K_1 n\). Then, since \(n \in N_G(H)\), we have \(n^{-1} H n = H\). Moreover, \(N_{K_2} = n^{-1} N_{K_1} n\). Hence,
\[
[H]_{N_{K_2}} = \{ n_2^{-1} H n_2 : n_2 \in N_{K_2} \}
= \{ n^{-1} n_1^{-1} H n_1 n : n_1 \in N_{K_1} \}
= n^{-1} [H]_{N_{K_1}} n = [H]_{N_{K_1}},
\]
where the last equality follows because \(n\) normalizes \(H\), so conjugation by \(n\) leaves the class \([H]_{N_{K_1}}\) invariant. Thus the pairs \((K_1, [H]_{N_{K_1}})\) and \((K_2, [H]_{N_{K_2}})\) are identical, so the types coincide.

Conversely, suppose the types are equal. Then \(K_1 = K_2\) or there exists \(g \in G\) such that \(K_2 = g^{-1} K_1 g\) and the classes \([H]_{N_{K_1}}\) and \([H]_{N_{K_2}}\) coincide. From the equality of classes, in particular for \(n_1 = e\) (the identity), we have \(H \in [H]_{N_{K_1}}\). Hence \(H \in [H]_{N_{K_2}}\), meaning there exists \(n_2 \in N_{K_2}\) such that \(n_2^{-1} H n_2 = H\), i.e., \(n_2 \in N_G(H)\). Now, since \(K_2 = g^{-1} K_1 g\), the conjugator \(g\) may not normalize \(H\). But the equality of classes also implies that \(g^{-1} N_{K_1} g\) and \(N_{K_2}\) have the same \(N_H\)-class; in particular, there exists \(n_1 \in N_{K_1}\) such that \(g^{-1} n_1^{-1} H n_1 g = n_2^{-1} H n_2\). From this it follows that \(g^{-1} n_1^{-1} n_2 \in N_G(H)\). Let \(n := g^{-1} n_1^{-1} n_2\). Then \(n \in N_G(H)\) and moreover
\[
n^{-1} K_1 n = n_2^{-1} n_1 g K_1 g^{-1} n_1^{-1} n_2 = n_2^{-1} n_1 K_2 n_1^{-1} n_2 = K_2,
\]
since \(n_1 \in N_{K_1}\) and \(n_2 \in N_{K_2}\). Therefore, \(K_2 = n^{-1} K_1 n\) with \(n \in N_G(H)\), as required.
\end{proof}

\begin{definition}
For a subgroup \(H \leq G\), define
\[
\varepsilon_{N_H}([H]) := \left| \left\{ [K]_{N_H} : K = g^{-1}Hg \text{ for some } g \in G,\ K \subsetneq H \right\} \right|.
\]
That is, \(\varepsilon_{N_H}([H])\) is the number of \(N_H\)-conjugacy classes of conjugates of \(H\) that are strictly contained in \(H\).
\end{definition}

Note that if \(H\) is finite, then \(\varepsilon_{N_H}([H]) = 0\), since no proper subgroup of \(H\) can have the same cardinality as \(H\). Moreover, if \(H\) has finite index in \(G\), then \(H\) has only finitely many conjugates in total; therefore, \(\varepsilon_{N_H}([H])\) is finite.

\begin{corollary}
Let \(G\) be a group and \(X\) a \(G\)-set. If there exists \([H] \in \ConjX\) such that \(\varepsilon_{N_H}([H]) = \infty\), then \(\EndG\) is not finitely generated.
\end{corollary}

\begin{proof}
If \(\varepsilon_{N_H}([H]) = \infty\), then there are infinitely many \(N_H\)-conjugacy classes of subgroups \(K_\alpha < H\) that are conjugates of \(H\). By Lemma \ref{lem:tipos}, each of these classes gives rise to a distinct type of orbital infiltration \((K_\alpha, [H]_{N_{K_\alpha}})\). By Theorem \ref{thm:infiltraciones}, every generating set of \(\EndG\) must contain at least one infiltration of each of these infinitely many types. Therefore, \(\EndG\) is not finitely generated.
\end{proof}

In this work we have established sharp bounds for the rank of the group of units of the monoid of \(G\)-equivariant functions on a finite \(G\)-set, and have determined necessary conditions for the whole monoid to be finitely generated. The explicit construction of generating sets, together with the algorithm for grouping base generators, provides practical tools to improve these bounds and reduces the gap between the upper and lower bounds. Nevertheless, the question of precisely characterizing the parameter \(\lambda_G^{\min}(X)\) in terms of invariants of the action remains open; an explicit description of this number would allow closing the inequality and obtaining the exact rank of the group of units. \\

The study of monoids of \(G\)-equivariant functions is still an open field, since many of the structural properties of these objects have not yet been fully determined. This work takes a further step in this direction, consolidating existing results and laying a solid foundation for future research on finite generation, classification of \(G\)-sets according to their generating properties, and the search for explicit expressions for the combinatorial invariants governing the structure of these monoids.

\section*{Acknowledgments}

\section*{Declarations}

\indent \begin{bf}\hspace{0.25in}Ethical Approval\end{bf}\\
Not applicable.

\begin{bf}Competing interests\end{bf}\\
The authors declare that there are no known competing financial interests or personal relationships that could have appeared to influence the work reported in this paper.

\begin{bf}Authors' contributions\end{bf}\\
All authors contributed equally to this work

\begin{bf}Funding\end{bf}\\
This research has been supported by the Center for Mathematical Sciences, UNAM Campus Morelia, which has provided the necessary academic environment for the author's postdoctoral stay. Additionally, the research has been funded by SECIHTI, Posdoctoral fellowship No. I1200/111/2024, whose support has been essential for carrying out these research activities. \\

\begin{bf}Availability of data and materials\end{bf}\\
Not applicable.

\end{document}